\documentclass[12pt,a4paper]{amsart}

\usepackage[latin1]{inputenc}
\usepackage{enumerate}
\usepackage{rotating}
\usepackage{fancybox}
\usepackage{float}
\restylefloat{table}
\usepackage{graphicx}

\usepackage{amsmath}

\usepackage{amssymb}

\usepackage{amsthm}

\usepackage[arrow, matrix, curve]{xy}

\usepackage{a4wide}

\usepackage{txfonts}

\usepackage{hyperref}

\usepackage{color}

\theoremstyle{plain}
\newtheorem{thma}{Main Theorem}
\newtheorem{thm}{Theorem}[section]
\newtheorem{prop}[thm]{Proposition}
\newtheorem{cor}[thm]{Corollary}

\theoremstyle{definition}
\newtheorem{defn}[thm]{Definition}
\newtheorem{lem}[thm]{Lemma}

\numberwithin{equation}{thm}

\DeclareMathAlphabet{\mathpzc}{OT1}{pzc}{m}{it}

\DeclareMathOperator{\Kopf}{top}
\DeclareMathOperator{\rad}{rad}
\DeclareMathOperator{\End}{End}
\DeclareMathOperator{\cx}{cx}
\DeclareMathOperator{\Dist}{Dist}
\DeclareMathOperator{\SL}{SL}
\DeclareMathOperator{\Ad}{Ad}
\DeclareMathOperator{\image}{Im}
\DeclareMathOperator{\charac}{char}
\DeclareMathOperator{\St}{St}
\DeclareMathOperator{\modu}{mod}
\DeclareMathOperator{\Hom}{Hom}
\DeclareMathOperator{\Ext}{Ext}
\DeclareMathOperator{\Id}{Id}

\DeclareMathOperator{\Lie}{Lie}

\begin{document}

\title{Representation type of Frobenius-Lusztig kernels}
\author{Julian K\"ulshammer}
\address{Christian-Albrechts-Universit\"at zu Kiel, Ludewig-Meyn-Str. 4, 24098 Kiel, Germany}
\email{kuelshammer@math.uni-kiel.de}
\thanks{Supported by the D.F.G. priority program SPP1388 "Darstellungstheorie"}

\begin{abstract}
In this article we show that almost all blocks of all Frobenius-Lusztig kernels are of wild representation type extending results of Feldvoss and Witherspoon, who proved this result for the principal block of the zeroth Frobenius-Lusztig kernel. Furthermore we verify the conjecture that there are infinitely many Auslander-Reiten components for a finite dimensional algebra of infinite representation type for selfinjective algebras whose cohomology satisfies certain finiteness conditions.
\end{abstract}

\maketitle

\section*{Introduction}

Around 1990, Lusztig defined a finite dimensional quantum group at a root of unity \cite{L90}, called the small quantum group, the restricted quantum group or the Frobenius-Lusztig kernel. It can be seen as a quantum analogue of the first Frobenius kernel of an algebraic group. Generalizing this finite dimensional algebra in characteristic $p$, Drupieski defined quantum analogues of higher Frobenius kernels, which he called  (higher) Frobenius-Lusztig kernels \cite{Dru11}.\\

An important question in the representation theory of finite dimensional algebras is, whether it is possible to classify all the indecomposable modules up to isomorphism. An algebra is said to be of finite representation type if there are only finitely many isoclasses, of tame representation type if it is not representation-finite and the isoclasses in each dimension are parameterized by finitely many one-parameter-families, and of wild representation type if the representation theory of this algebra is at least as complicated as that of any other algebra (see \cite{D80}).\\

In 2009, Feldvoss and Witherspoon provided a generalization of Farnsteiner's Theorem, that certain classes of algebras are wild if there is a module of complexity greater than 2. They applied it to prove that the small quantum group is wild in most cases. However a wild algebra can in general have tame or representation finite blocks. In this paper we prove that this cannot happen for the small quantum group:

\renewcommand{\thethma}{\thickspace \hspace{-0.5em}}

\begin{thma}
Let $\mathfrak{g}$ be a finite dimensional simple complex Lie algebra. Then with the exception of the simple Steinberg block, all blocks of the small quantum group are wild if $\mathfrak{g}\neq \mathfrak{sl}_2$.
\end{thma}

We also generalize this statement to higher Frobenius-Lusztig kernels. In some cases we have to employ different methods to prove this statement and (partially) compute the quiver and relations of these algebras.\\

A related topic is the beginning of the study of the Auslander-Reiten theory of small quantum groups, that we want to investigate in more detail in a forthcoming paper. Auslander-Reiten theory provides a way to partially understand also the representation theory of wild algebras by factoring homomorphisms into irreducible ones (for an introduction, see \cite{ARS95}, \cite{ASS06} or \cite{B95}). To visualize this factorization, a (usually infinite) quiver is associated to an algebra. In \cite[p. 15]{Rei82}, Reiten asks whether a connected Auslander-Reiten quiver yields finite representation type. A well-known stronger conjecture states that there are infinitely many such components if the algebra is of infinite representation type. This conjecture is known to hold if the algebra is tame \cite[Corollary F (1)]{CB88}, if the algebra is hereditary (cf. \cite[Conjectures (3)]{ARS95}), or if the algebra is a block of a cocommutative Hopf algebra whose dual is local \cite[Corollary 3]{Fa00}. Here we note that the proof of the last result generalizes to selfinjective algebras whose cohomology satisfies certain finiteness conditions, i.e. to algebras where support varieties can be defined. In particular this statement holds for the small quantum group, its Borel and nilpotent part. The latter is not a Hopf algebra.\\

The paper is organized as follows: In Section 1 we prove some facts about extensions of modules, when there is a Hopf algebra with a normal subalgebra. In Section 2 we verify the conjecture that there are infinitely many Auslander-Reiten components. In Section 3 we recall some general statements on Frobenius-Lusztig kernels. In Section 4 we are concerned with the representation type of the Borel and nilpotent part of the quantum group and in the final section we prove the main theorem and some generalizations.\\

Throughout the paper let $k$ be an algebraically closed field. All algebras and modules are assumed to be finite-dimensional unless they are obviously not. A tensor product $\otimes$ without supscript will always refer to the tensor product over the field $\otimes_k$.

\section{Extensions of Hopf algebras}

The statements in this section (with the exception of Proposition \ref{specialblock}) are mostly due to Farnsteiner in \cite[Section 2]{Fa06}, who proved them for cocommuative Hopf algebras. If one switches the orders of the tensor products at the appropriate places, his proofs apply almost verbatim. We include them for convenience of the reader. For general results on Hopf algebras, that we use without mentioning, we refer to \cite{Kas95} and \cite{Mon93}.\\

For a Hopf algebra $A$, we denote its counit and antipode by $\varepsilon_A$ and $S_A$, respectively. If the algebra is obvious from the context, we sometimes omit the subscript. If $A$, $B$ and $C$ are augmented algebras by a weak Hopf sequence $k\to B\to A\to C\to k$ we understand that $A$ is a Hopf algebra, $B$ is a normal augmented subalgebra of $A$, such that $A/A \ker \varepsilon_B\cong C$ as augmented algebra. Here an augmented subalgebra is called normal if $A\ker \varepsilon_B$ is also a right ideal. In this case we also write $A//B\cong C$. If a $C$-module $M$ is viewed as an $A$-module, we denote it by $M^{[1]}$. This notation is now (beside the notation $M^{[\ell]}$) standard for quantum groups and comes from the notation in the context of algebraic groups. Conversely, an $A$-module $M$ on which $B$ acts trivially is denoted by $M^{[-1]}$ if viewed as a $C$-module. Furthermore an $A$-module $M$ which is viewed as a $B$-module will be denoted $M|_B$.

\begin{lem}\label{simples}
Let $k\to B\to A\to C\to k$ be a weak Hopf sequence. If $L_1,\dots,L_n$ are simple $A$-modules, such that $\{L_1|_B,\dots,L_n|_B\}$ is a complete set of representatives for the isoclasses of simple $B$-modules, then every simple $A$-module $S$ is of the form $S\cong M^{[1]}\otimes L_i$ for a unique $i\in \{1,\dots,n\}$ and a (up to isomorphism) unique simple $C$-module $M$.
\end{lem}

\begin{proof}
Let $S$ be a simple $A$-module. Then by assumption there exists $i\in \{1,\dots,n\}$, such that $L_i|_B\hookrightarrow S|_B$. Hence the $A$-linear map $\varphi: \Hom_B(L_i,S)\otimes L_i\to S, f\otimes x\mapsto f(x)$ is non-zero and hence surjective. The image is the $L_i$-isotypical component of $S|_B$. By Schur's Lemma (cf. \cite[I.2.14 (3)]{Jan03}) the two modules also have the same dimensions and therefore, the map is an isomorphism. Thus $S\cong \Hom_B(L_i,S)\otimes L_i$ and $\Hom_B(L_i,S)$ is a simple $C$-module.\\
The unicity is also proven: If a simple module $S$ is isomorphic to $M^{[1]}\otimes N$ with $M$ a simple $C$-module and $N$ a simple $B$-module, then $N$ is uniquely determined by $S|_B\cong N^{\dim M}$ and $M\cong \Hom_B(N,S)$ is also uniquely determined.
\end{proof}

\begin{lem}\label{extsimpleextensionHopf}
Let $k\to B\to A\to C\to k$ be a weak Hopf sequence. Suppose $\Ext^1_B(S,S)=0$ for every simple $A$-module $S$. If $L_1,L_2$ are simple $A$-modules and $M_1,M_2$ are simple $C$-modules such that $L_i|_B$ is simple, then we have
\[\Ext^1_A(M_1^{[1]}\otimes L_1, M_2^{[1]}\otimes L_2)\cong \begin{cases}\Ext^1_C(M_1,M_2)&\text{if}\thickspace L_1\cong L_2\\\Hom_C(M_1,M_2\otimes \Ext^1_B(L_1,L_2)^{[-1]})&\text{if}\thickspace L_1\ncong L_2.\end{cases}\]
\end{lem}

\begin{proof}
Let $M$ be a $C$-module and let $E$ and $V$ be $A$-modules. Then we have isomorphisms of $A$-modules $\Hom_k(M^{[1]}\otimes E,V)\cong V\otimes (M^{[1]}\otimes E)^*\cong V\otimes (E^*\otimes (M^{[1]})^*)\cong (V\otimes E^*)\otimes (M^{[1]})^*\cong \Hom_k(M^{[1]},V\otimes E^*)$.\\
Taking $A$-invariants one gets: $\Hom_A(M^{[1]}\otimes E,V)\cong (\Hom_k(M^{[1]}\otimes E,V))^A\cong ((\Hom_k(M^{[1]},V\otimes E^*))^B)^C\cong (\Hom_k(M^{[1]},(V\otimes E^*)^B))^C\cong \Hom_C(M,\Hom_B(E,V)^{[-1]})$.\\
This shows that $\Hom_B(E,-)$ is right adjoint to $-\otimes E$ as a functor $\modu A\to \modu C$. Therefore \cite[Proposition 2.3.10]{Wei94} tells us that it sends injective modules to injectives, as $-\otimes E$ is exact.\\
Now $\Hom_C(M,-)$ is left exact and $\Hom_B(E,-)$ maps injectives to injectives, in particular to acyclic ones for $\Hom_C(M,-)$. Therefore by \cite[4.1 (1)]{Jan03} there is a Grothendieck spectral sequence $\Ext_C^n(M,\Ext^m_B(E,V)^{[-1]})\Rightarrow \Ext_A^{n+m}(M^{[1]}\otimes E,V)$. This gives rise to a 5-term exact sequence (see e.g. \cite[4.1 (4)]{Jan03}), for which we only display the first four terms:
\begin{align*}
0&\to \Ext^1_C(M,\Hom_B(E,V)^{[-1]})\to \Ext^1_A(M^{[1]}\otimes E,V)\\
&\to \Hom_C(M,\Ext^1_B(E,V)^{[-1]})\to \Ext^2_C(M,\Hom_B(E,V)^{[-1]}).
\end{align*}
Now we set $M:=M_1, E:=L_1$ and $V:=M_2^{[1]}\otimes L_2$. Furthermore we recognize that there are natural isomorphisms $\Hom_k(L_1,M_2^{[1]}\otimes L_2)\cong (M_2^{[1]}\otimes L_2)\otimes L_1^*\cong M_2^{[1]}\otimes (L_2\otimes L_1^*)\cong M_2^{[1]}\otimes \Hom_k(L_1, L_2)$. Again we take $B$-invariants to arrive at $\Hom_B(L_1,M_2\otimes L_2)\cong M_2\otimes \Hom_B(L_1,L_2)$. Now we have that $M_2^{[1]}\otimes -$ is exact, so by \cite[4.1 (2)]{Jan03} this implies $\Ext^n_B(L_1,M_2^{[1]}\otimes L_2)\cong M_2^{[1]}\otimes \Ext^1_B(L_1,L_2)$. Thus the 5-term sequence reads as:
\begin{align*}
0&\to \Ext^1_C(M_1,M_2\otimes \Hom_B(L_1,L_2)^{[-1]})\to \Ext^1_A(M_1^{[1]}\otimes L_1,M_2^{[1]}\otimes L_2)\\
&\to \Hom_C(M_1,M_2\otimes \Ext^1_B(L_1,L_2)^{[-1]})\to \Ext^2_C(M_1,M_2\otimes \Hom_B(L_1,L_2)^{[-1]}).
\end{align*}
Therefore if $L_1\ncong L_2$, then $\Ext^1_A(M_1^{[1]}\otimes L_1,M_2^{[1]}\otimes L_2)\cong \Hom_C(M_1,M_2\otimes \Ext^1_B(L_1,L_2)^{[-1]})$ as $M_2\otimes \Hom_B(L_1,L_2)=0$ by Schur's Lemma.\\
If on the other hand $L_1\cong L_2$, then by Schur's Lemma $\Hom_B(L_i|_B,L_i|_B)$ and $\Hom_A(L_i,L_i)$ are both $1$-dimensional. Thus they are eqal and hence $C$ acts trivially on this space. Therefore $\Ext^1_A(M_1^{[1]}\otimes L_1,M_2^{[1]}\otimes L_2)\cong \Ext^1_C(M_1,M_2\otimes \Hom_B(L_1,L_2)^{[-1]})\cong \Ext^1_C(M_1,M_2)$ as additionally by assumption $M_2\otimes \Ext^1_B(L_1,L_1)=0$.
\end{proof}

The following proposition is an analogue of \cite[Proposition 3.11]{AM11}, where it is stated for quantum category $\mathcal{O}$ in characteristic zero. We borrow their notation and speak of  the special block.

\begin{prop}\label{specialblock}
Let $k\to B\to A\to C\to k$ be a weak Hopf sequence. Let $L_1,\dots,L_n$ be simple $A$-modules, such that $\{L_1|_B,\dots,L_n|_B\}$ is a complete set of representatives for the simple $B$-modules. Furthermore let $L_n|_B$ be projective. Then the simple modules of the form $M^{[1]}\otimes L_n$ form a block ideal $\mathcal{B}^{spec}$ in $A$ (i.e. a direct sum of blocks), that is equivalent to $\modu C$.
\end{prop}

\begin{proof}
We define two functors $F:\modu A\to \modu C, V\mapsto \Hom_B(L_n,V)^{[-1]}$ and $G: \modu C\to \modu A, M\mapsto M^{[1]}\otimes L_n$. Both functors are exact since $L_n|_B$ is projective and $k$-flat. Furthermore they are adjoint since $\Hom_A(M^{[1]}\otimes E,V)\cong \Hom_C(M,\Hom_B(E,V)^{[-1]})$, compare the proof of the foregoing lemma.\\
We have $\Hom_B(L_n, M^{[1]}\otimes L_n)\cong M^{[1]}\otimes \End_B(L_n)\cong M^{[1]}$, i.e. $F\circ G\cong \Id$. Also $G\circ F\cong \Id$, since this holds for the simple modules in the special block by the proof of Lemma \ref{simples}. Now apply the adjunction to get a map in $\Hom_A(G\circ F M,M)$. Induction using the $5$-Lemma tells us that this map is an isomorphism. Thus we have an equivalence of categories.\\
To prove that it is a block ideal apply the foregoing lemma to get:
\[\Ext^1_A(M^{[1]},L_n, N^{[1]}\otimes L_r)\cong \Hom_C(M,N\otimes \Ext^1_B(L_n,L_r)^{[-1]})=0\]
for $r\neq n$.
\end{proof}

In a similar fashion one can also generalize \cite[Proposition 2.3]{Fa06}:

\begin{prop}\label{extprojectiveextensionHopf}
Let $k\to B\to A\to C\to k$ be a weak Hopf sequence. Then for $A$-modules $X,Y$ and a projective $C$-module $Q$, one has isomorphisms $\Ext^j_A(Q^{[1]}\otimes X, N^{[1]}\otimes Y)\cong \Hom_C(Q^{[1]},N\otimes \Ext^j_B(X,Y)^{[-1]})$ for any $C$-module $N$ and all $j\in \mathbb{N}$.\\
Suppose that every simple $B$-module is the restriction of an $A$-module. If $P_1,\dots,P_n$ are $A$-modules, such that $\{P_1|_B,\dots,P_n|_B\}$ is a complete set of representatives of projective indecomposable $B$-modules and $\{Q_1,\dots, Q_m\}$ is a complete set of representatives of the isoclasses of the projective indecomposable $C$-modules, then the $A$-modules $Q_r^{[1]}\otimes P_i$ form a complete system of projective indecomposable $A$-modules.
\end{prop}

\section{Infinitely many Auslander-Reiten components}

In this section we prove that there are infinitely many Auslander-Reiten components for certain algebras where support varieties can be defined. The approach is modelled on \cite{Fa98}. In the proof we use slightly different combinatorics. Recall the following definitions:

\begin{defn}[cf. \cite{FW09}, \cite{EHTSS04}]
\begin{enumerate}[(i)]
\item A Hopf algebra $A$ is said to satisfy (fg) if the following two conditions are satisfied:
\begin{enumerate}[({fg}1)]
\item $H^{ev}(A,k)$ is finitely generated.
\item $\Ext^\bullet_A(M,N)$ is finitely generated as a module for $H^{ev}(A,k)$.
\end{enumerate}
\item A selfinjective algebra $A$ is said to satisfy (Fg) if the following two conditions are satisfied:
\begin{enumerate}[({Fg}1)]
\item There exists a graded subalgebra $H$ of the Hochschild cohomology ring $HH^\bullet(A)$ such that $H$ is a commutative Noetherian ring and $H^0=HH^0(A)=Z(A)$.
\item $\Ext^\bullet_A(A/\rad{A}, A/\rad{A})$ is finitely generated as an $H$-module.
\end{enumerate}
\end{enumerate}
\end{defn}

By $\gamma$ we denote the rate of growth of a graded vector space. The complexity $\cx_A$ of a module is defined to be the rate of growth of a minimal projective resolution. To the author's knowledge in the following, (ii) is not stated explicitely anywhere for (fg)-Hopf algebras, but all the statements and proofs for group algebras transfer verbatim. The theorem is proven by associating a certain variety to a module, whose dimension equals the complexity of the module, the so called support variety.

\begin{thm}[(i): {\cite[p. 16]{Bro98}}, {\cite[Theorem 2.5]{EHTSS04}}, (ii): {\cite[Proposition 5.2, Theorem 5.3]{EHTSS04}}, cf. {\cite[Corollary 5.10.3]{B91}}]\label{cxdimsupporthochschild}
Let $A$ be an (fg)-Hopf algebra or a selfinjective algebra satisfying (Fg).
\begin{enumerate}[(i)]
\item We have $\cx_A(M)=\gamma(\Ext^\bullet(M,M))<\infty $ for any module $M$ in $\modu A$.
\item An indecomposable module $M$ is periodic iff $\cx_A M=1$.
\end{enumerate}
\end{thm}

Let us now recall some notions of Auslander-Reiten theory. The Auslander-Reiten quiver of an algebra $A$ is the oriented graph which has as vertices the isomorphism classes of indecomposable modules and as arrows the irreducible maps, i.e. the maps $f$ which are neither split monomorphisms nor split epimorphisms and whenever $f=f_2f_1$, then $f_1$ is a split monomorphism or $f_2$ is a split epimorphism. The stable Auslander-Reiten quiver $\Gamma_s(A)$ is obtained by omitting the vertices corresponding to projective modules and all adjacent arrows. For a (connected) component $\Theta$ of the stable Auslander-Reiten quiver, let $\Theta(d):=\{[M]\in \Theta|\dim M\leq d\}$. For a selfinjective algebra, the Auslander-Reiten translation is denoted by $\tau_A$, the szyzygy functor by $\Omega_A$, and the Nakayama functor by $\mathcal{N}_A$ (see e.g. \cite[Remark after Proposition 4.12.9]{B95}).

\begin{lem}[cf. {\cite[Lemma 1.1]{Fa98}}, {\cite[Lemma 3.1]{Fa00}}]
Let $A$ be an (fg)-Hopf algebra or an (Fg)-selfinjective algebra and let $\Theta\subseteq \Gamma_s(A)$ be a component. If $\mathcal{T}\subseteq \Theta$ is a $\tau$-orbit, then $\mathcal{T}\cap \Theta(d)$ is finite for every $d>0$.
\end{lem}

\begin{proof}
Assume there is $d>0$, such that $\mathcal{T}\cap \Theta(d)$ is infinite. Then $\Theta$ is obviously a non-periodic component. By assumption there is a vertex $[M]\in \Theta(d)$ such that $\dim \tau_A^j(M)\leq d$ for $j$ belonging to some infinite set $J\subset \mathbb{Z}$. As $\tau_A=\Omega^2_A\circ \mathcal{N}_A$ we also have $\dim \Omega^{2j}_A(M)\leq d'$ for $j$ belonging to $J$, since $\mathcal{N}_A$ is an exact functor permuting the isoclasses of the simples. Without loss of generality assume that $d'=d$ and that $0\in J$. Consequently
\[\dim \Ext^{2j}_A(M,M)\leq \dim \Hom_A(\Omega^{2j}_A(M),M)\leq d^2\thickspace \text{for all}\thickspace j\in J\cap \mathbb{N}_0\]
while
\[\dim \Ext^{-2j}_A(M,M)\leq \dim \Hom_A(M,\Omega^{2j}_A(M))\leq d^2\thickspace \text{for all}\thickspace j\in J\cap -\mathbb{N}_0.\]
Hence there are infinitely many $n\in \mathbb{N}_0$, such that $\dim \Ext^{2n}_A(M,M)\leq d^2$. On the other hand since $\Theta$ is not periodic, it follows from the foregoing theorem that $\cx_A M \geq 2$, hence $\dim \Ext^{2n}_A(M,M)\geq Cn$ for some $C>0$ and all $n\geq 0$, a contradiction.
\end{proof}

The following result relies on the fact, that for a selfinjective algebra whose modules all have finite complexity the components of a stable Auslander-Reiten quiver have a very special shape. They are constructed as follows: Start with an (unoriented) graph $T$. Then orient it in some way. Then define a new quiver with vertices $\mathbb{Z}\otimes T_0$, where $T_0$ is the vertex set of $T$ and arrows $(i,t)\to (i,t')$ and $(i,t')\to (i+1,t)$ for every arrow $t\to t'$ in $T$ and every integer $i$. This new quiver is denoted by $\mathbb{Z}[T]$, the Auslander-Reiten translation is defined by $\tau((i,t))=(i-1,t)$.

\begin{thm}[cf. {\cite[Theorem 1.2]{Fa98}}, {\cite[Theorem 3.2]{Fa00}}]\label{thetafinite}
Let $A$ be an (fg)-Hopf algebra or an (Fg)-selfinjective algebra and let $\Theta\subseteq \Gamma_s(A)$ be a component. Then $\Theta(d)$ is finite for every $d>0$.
\end{thm}

\begin{proof}
If $\Theta$ is periodic, this result was obtained by Bautista and Coelho (cf. \cite[Proposition 5.4]{Liu96}). If $\Theta$ is nonperiodic, then by the version of Webb's Theorem proved by Kerner and Zacharia (\cite[Main Theorem]{KZ11}) $\Theta\cong \mathbb{Z}[\Delta]$, where $\Delta$ is a Euclidean or infinite Dynkin diagram. If $\Delta$ is Euclidean, then $\mathbb{Z}[\Delta]$ possesses only finitely many $\tau$-orbits, and the statement follows immediately from the foregoing lemma. If $\Delta$ is $\mathbb{A}_\infty^\infty$ or $\mathbb{D}_\infty$ the statement follows from the proof of \cite[Proposition]{MR88}. Therefore it suffices to prove the statement for $\Theta\cong \mathbb{Z}[\mathbb{A}_\infty]$, obtained by a linearly oriented half-line. We label the vertices of $\Theta$ by $(a,b)$, where $a\in \mathbb{Z}$ as above and $b\in \mathbb{N}$:\\
Since there are only finitely many meshes in $\Theta$ that are associated with projective modules, there exists a vertex $(a,b)$, such that all meshes associated with projective modules occur inside the region $\tilde{\Upsilon}(a,b):=\{(n,m)\thickspace |\thickspace n+m\leq a+b, n\geq a\}$, which is highlighted red in the following diagram:
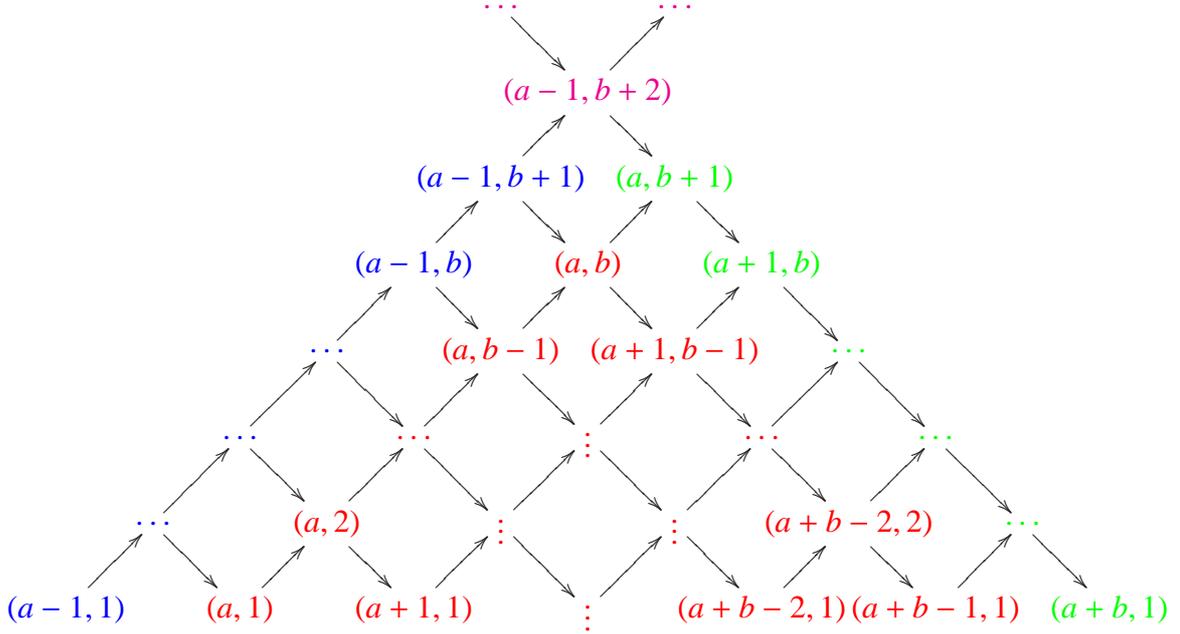
\[\begin{xy}
\xymatrix@!=0.3cm{
&&&&&{\color{magenta}\dots}\ar[rd]&&{\color{magenta}\dots}\\
&&&&&&{\color{magenta}(a-1,b+2)}\ar[ru]\ar[rd]\\
&&&&&{\color{blue}(a-1,b+1)}\ar[ru]\ar[rd]&&{\color{green}(a,b+1)}\ar[rd]\\
&&&&{\color{blue}(a-1,b)}\ar[ru]\ar[rd]&&{\color{red}(a,b)}\ar[ru]\ar[rd]&&{\color{green}(a+1,b)}\ar[rd]\\
&&&{\color{blue}\dots}\ar[ru]\ar[rd]&&{\color{red}(a,b-1)}\ar[ru]\ar[rd]&&{\color{red}(a+1,b-1)}\ar[ru]\ar[rd]&&{\color{green}\dots}\ar[rd]\\
&&{\color{blue}\dots}\ar[ru]\ar[rd]&&{\color{red}\dots}\ar[ru]\ar[rd]&&{\color{red}\vdots}\ar[ru]\ar[rd]&&{\color{red}\dots}\ar[ru]\ar[rd]&&{\color{green}\dots}\ar[rd]\\
&{\color{blue}\dots}\ar[ru]\ar[rd]&&{\color{red}(a,2)}\ar[ru]\ar[rd]&&{\color{red}\vdots}\ar[ru]\ar[rd]&&{\color{red}\vdots}\ar[ru]\ar[rd]&&{\color{red}(a+b-2,2)}\ar[ru]\ar[rd]&&{\color{green}\dots}\ar[rd]\\
{\color{blue}(a-1,1)}\ar[ru]&&{\color{red}(a,1)}\ar[ru]&&{\color{red}(a+1,1)}\ar[ru]&&{\color{red}\vdots}\ar[ru]&&{\color{red}(a+b-2,1)}\ar[ru]&&{\color{red}(a+b-1,1)}\ar[ru]&&{\color{green}(a+b,1)}\\
}
\end{xy}\]
A mesh outside this region looks as follows:
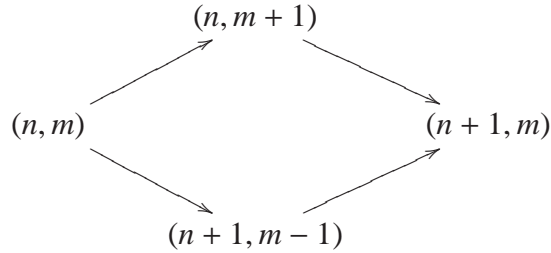
\[\begin{xy}
\xymatrix{
&(n,m+1)\ar[rd]\\
(n,m)\ar[ru]\ar[rd]&&(n+1,m)\\
&(n+1,m-1)\ar[ru]
}
\end{xy}\]
Since it corresponds to an exact sequence, we have $\dim (n,m+1)-\dim (n+1,m)=\dim (n,m)-\dim (n+1,m-1)$ and since every irreducible morphism is either injective or surjective the surjectivity of the map $(n,m)\to (n+1,m-1)$ implies the surjectivity of the map $(n,m+1)\to (n+1,m)$. But since the map $(n,2)\to (n+1,1)$ is surjective for all $n<a$ this implies the surjectivity of all maps $(n,m)\to (n+1,m-1)$ for all $n<a$, i.e. in the blue and magenta region. Dually the injectivity of the map $(n+1,m-1)\to (n+1,m)$ implies the injectivity of the map $(n,m)\to (n,m+1)$ and since the map $(n,1)\to (n,2)$ is injective, for all $n+m>a+b$ we have the injectivity of all maps $(n,m)\to (n,m+1)$ for $n+m>a+b$, i.e. in the green and magenta region.\\
Therefore for $(n,m)$ with $n<a$ and $n+m\leq a+b$ we have that $\dim (n,m)\geq \min\{a-n+1,m\}$. Now assume that $(n,m)\in \Theta(d)$. Then $d\geq a-n+1$ together with $n+m\leq a+b$ implies that $m\leq a+b-n\leq b+d-1$. Thus there are no $(n,m)\in \Theta(d)$ inside the blue region with $m>b+d-1$.\\
Dually for $(n,m)$ with $n\geq a$ and $n+m>a+b$ we have that $\dim (n,m)\geq \min\{n+m-(a+b),m\}$. Again assume $(n,m)\in \Theta(d)$. Then $d\geq n+m-(a+b)$ together with $n\geq a$ implies that $m\leq d-n+(a+b)\leq d+b$. Thus there are no $(n,m)\in \Theta(d)$ inside the green region with $m>b+d$.\\
In the magenta region we have that $\dim (n,m)\geq m-(b+1)$ and therefore also in this region for $m>d+b-1$ there are no $(n,m)\in \Theta(d)$.\\
Consequently $\Theta(d)\subseteq \{(n,m)|m<b+d-1\}$ is contained in a finite union of $\tau$-orbits, a contradiction to the foregoing lemma.
\end{proof}

Ringel (\cite{Ri96}) showed that Theorem \ref{thetafinite} does not hold for arbitrary selfinjective algebras.

\begin{cor}[{\cite[Corollary 3.3]{Fa00}}]
There are infinitely many components for a block of an (fg)-Hopf algebra or an (Fg)-self injective algebra, which is not of finite representation type.
\end{cor}

\begin{proof}
This follows from the foregoing theorem in conjunction with the second Brauer-Thrall conjecture, which is a theorem in the case of a finite dimensional algebra over an algebraically closed field (cf. \cite[IV.5 Conjecture 2]{ASS06}), that states that for a representation-infinite algebra $A$, there is a dimension $d$, such that there are infinitely many nonisomorphic $A$-modules of dimension $d$.
\end{proof}

\section{General facts on Frobenius-Lusztig kernels}

Here we recall some general facts on Frobenius-Lusztig kernels and fix our notation. It can be found in \cite{Jan96} and \cite{Dru11}.\\
For an indeterminate $q$ let us first recall the notion of $q$-binomial coefficients: $[n]:=\frac{q^n-q^{-n}}{q-q^{-1}}$, $[n]!:=[n]\cdots [2][1]$ and $\begin{bmatrix}m\\n\end{bmatrix}:=\frac{[m]!}{[n]![m-n]!}$. If we specialize $q$ to an element of the field, by abuse of notation we will use the same symbols to denote the elements of the field obtained in that way.\\
Let $\mathfrak{g}$ be a finite dimensional complex simple Lie algebra. Let $\mathbb{U}_\mathbb{Q}(\mathfrak{g})$ be the quantized universal enveloping algebra defined by Drinfeld and Jimbo (cf. \cite[4.3]{Jan96}) with generators $E_i, F_i, K_i^{\pm 1}$ over $\mathbb{Q}(q)$. For $\mathsf{A}:=\mathbb{Z}[q,q^{-1}]$ Lusztig defined an integral form $U_\mathsf{A}(\mathfrak{g})$. For $k$ a field and $\zeta$ a primitive $\ell$-th root of unity, one can extend scalars to $k$ by specializing $q$ to $\zeta$ and define $U_k(\mathfrak{g}):=U_\mathsf{A}(\mathfrak{g})\otimes_{\mathsf{A}} k$. The elements $E_i\otimes 1$, $F_i\otimes 1$ and $K_i^{\pm 1}\otimes$ will by standard abuse of notation again be denoted by $E_i$, $F_i$ and $K_i^{\pm 1}$, respectively. Because one can recover the whole integral representation theory of $U_k(\mathfrak{g})$ from the factor algebra $U_\zeta(\mathfrak{g}):=U_k(\mathfrak{g})/\langle K_i^\ell-1\rangle$ (it corresponds to looking at modules of type $\mathbf{1}$) we confine our attention to this algebra.\\
The small quantum group $u_\zeta(\mathfrak{g})$ is now defined as the subalgebra generated by the images of the elements $E_i,F_i,K_i^{\pm 1}$ under the canonical projection, in the remainder again denoted by the same symbols.\\
If $\charac k=p>0$, then there are higher Frobenius-Lusztig kernels $U_\zeta(G_r)$, defined as the subalgebra of $U_\zeta(\mathfrak{g})$ generated by $E_i, E_i^{(p^j\ell)}, F_i, F_i^{(p^j\ell)}, K_i^{\pm 1}$, where $j<r$ and $E_i^{(n)}$ and $F_i^{(n)}$ denote the specializations of the divided powers $E_i^{(n)}:=\frac{E_i^n}{[n]!}$ and $F_i^{(n)}:=\frac{F_i^n}{[n]!}$. The name derives from the following weak Hopf sequence, where $\mathfrak{g}=\Lie(G)$, $G_r$ denotes the $r$-th Frobenius kernel of the algebraic group $G$ and $\Dist(G_r)$ denotes the algebra of distributions of $G_r$:
\[k\to u_\zeta(\mathfrak{g})\to U_\zeta(G_r)\to \Dist(G_r)\to k,\]
where the surjective map is the quantum Frobenius homomorphism defined by Lusztig. We will also use the notation $U_\zeta(G_0):=u_\zeta(\mathfrak{g})$.\\
We will replace $\mathfrak{g}$ by $\mathfrak{b}$, $\mathfrak{n}$ or $\mathfrak{t}$ if we just look at the subalgebras coming from the $E_i$'s and the $K_i$'s, the $E_i$'s, or the $K_i$'s respectively.\\
The weak Hopf sequence then restricts, e.g. for the Borel subgroup $B_r$ of $G_r$ we get a weak Hopf sequence
\[k\to u_\zeta(\mathfrak{b})\to U_\zeta(B_r)\to \Dist(B_r)\to k.\]
The $\mathfrak{b}$- and $\mathfrak{n}$-quantum groups are related by a smash product construction, i.e. $U_\zeta(B_r)=U_\zeta(N_r)\# U_\zeta(T_r)$. This means the following: Let $H$ be a Hopf algebra and let $A$ be an algebra, which is also an $H$-module via $h\otimes a\mapsto h\cdot a$. Then $h\cdot (aa')=\sum_{(h)}(h_{(1)}\cdot a)(h_{(2)}\cdot a')$ and $h\cdot 1_A=\varepsilon(h)1_A$ for all $a,a'\in A$ and $h\in H$, where we use Sweedler notation $\Delta(h)=\sum_{(h)}h_{(1)}\otimes h_{(2)}$. Then $A$ is called an $H$-module algebra. The smash product $A\# H$ is then defined as the algebra with underlying vector space $A\otimes H$, whose simple tensors are denoted by $a\# h$ and multiplication $(a\# h)(b\# h'):= \sum_{(h)}a(h_{(1)}\cdot b)\# h_{(2)}h'$. This gives us another weak Hopf sequence
\[k\to U_\zeta(N_r)\to U_\zeta(B_r)\to U_\zeta(T_r)\to k.\]

We need the following notation from Lie theory. We denote by $\Phi, \Phi^\vee, \Pi, \Phi^+, X^+$ the set of roots, coroots, simple roots, positive roots, dominant weights respectively. There is a unique positive definite symmetric bilinear form $(-,-)$ on $\mathbb{R}\Phi$, such that $(\alpha,\alpha)=2$. Define $\Phi_\lambda:=\{\alpha\in \Phi| (\lambda+\rho,\alpha^\vee)\in \ell\mathbb{Z}\}$ and $\Phi_\lambda^+:=\Phi_\lambda\cap \Phi^+$. Furthermore for $n\in \mathbb{N}$ denote by $X_n:=\{\lambda\in X^+|0\leq (\lambda,\alpha^\vee)<n\forall \alpha\in \Pi\}$ the set of $n$-restricted weights. The half-sum of the positive roots is denoted by $\rho$.\\

The algebras $U_\zeta(N_r)$ are all local. The simple $u_\zeta(\mathfrak{b})$-modules are all $1$-dimensional and lift to $U_\zeta(B)$, therefore in particular to all of the $U_\zeta(B_r)$. Therefore by Lemma \ref{simples} every simple $U_\zeta(B_r)$-module is $1$-dimensional and of the form $\mu^{[1]}\otimes \lambda$, where $\mu\in X_p^+$ and $\lambda\in X_\ell^+$. One can also use the weak Hopf sequence relating $U_\zeta(N_r)$ and $U_\zeta(B_r)$ to conclude that the simple $U_\zeta(B_r)$ are $1$-dimensional: The trivial representation of $U_\zeta(N_r)$ lifts to $U_\zeta(B_r)$ since this is a Hopf algebra. Therefore all modules are tensor products of the $1$-dimensional representations of $U_\zeta(T_r)$ and the trivial representation. The simple $u_\zeta(\mathfrak{g})$-modules are in general not $1$-dimensional, but they also lift to $U_\zeta(G)$, thus to all $U_\zeta(G_r)$, i.e. again Lemma \ref{simples} is applicable and also all simple $U_\zeta(G_r)$-modules are of the form $ \overline{L}(\mu)^{[1]}\otimes L(\lambda)$, where $\mu\in X_ {p^r}^+$ and $\lambda\in X_\ell^+$ (cf. \cite[Theorem 3.1]{AM11}).\\

For each of the Frobenius-Lusztig kernels $U_\zeta(G_r)$ there is a projective simple module, denoted by $\St_{p^r\ell}$. In characteristic zero, $\St_\ell$ is a projective module for $U_\zeta(G)$. However in positive characteristic $\St_{\ell}$ will not be projective as a module for $U_\zeta(G_r)$.  Proposition \ref{extprojectiveextensionHopf} provides a conceptional explanation for this fact: For $\charac k=p>2h-2$, the projective $u_\zeta(\mathfrak{g})$-modules lift to $U_\zeta(G)$ (they are restrictions of tilting modules, see \cite[p. 15]{And04}),  the algebra of distributions is not semisimple and therefore the trivial $\Dist(G_r)$-module is not projective.\\

For our determination of the representation type of the Frobenius-Lusztig kernels we also need certain indecomposable modules, the quantized induced modules $H^0(\lambda)$ studied in \cite{APW91}, \cite{APW92} and \cite{AW92}. The indecomposability can be derived from \cite[Theorem 3.4]{APW92} and \cite[Proof of Theorem 1.10]{AW92}. The lower bound for their complexity under the weakest known assumptions is stated in \cite{BNPP11}. This is where the restrictions of our parameters in the theorems come from.\\

Applying Proposition \ref{specialblock} to the weak Hopf sequence $k\to u_\zeta(\mathfrak{g})\to U_\zeta(G_r)\to \Dist(G_r)\to k$ and the simple $u_\zeta(\mathfrak{g})$-projective Steinberg module $\St_\ell$ we get that the modules whose composition factors belong to $\ell X+ (\ell-1)\rho$ form a block ideal $\mathcal{B}^{spec}$ equivalent to $\Dist(G_r)$.\\

An algebra $A$ is called symmetric, if $A$ is isomorphic to its dual as a two-sided module. A further application of Lemma \ref{extsimpleextensionHopf} is given by the following:

\begin{prop}
The Frobenius-Lusztig kernels $U_\zeta(G_r)$ are symmetric algebras.
\end{prop}

\begin{proof}
It is well-known that $u_\zeta(\mathfrak{g})$ has a unique one-dimensional module: This can be derived from the fact that the Weyl group acts on the set of weights for a $U_\zeta(\mathfrak{g})$-module $M$ (\cite[1.7]{AW92}). The same is known for $\Dist(G_r)$ (cf. \cite[1.19 (1)]{Jan03}). Thus there is also a unique one-dimensional module for $U_\zeta(G_r)$ by Lemma \ref{extsimpleextensionHopf}. Since the antipode squared for $U_\zeta(G_r)$ is inner (see \cite[4.9 (1)]{Jan96}), the result follows (see e.g. \cite[4.6 Remark 2)]{Gor01}).
\end{proof}

\section{Representation type of half-quantum groups}

We are now concerned with the main part of this paper, the determination of the representation type of Frobenius-Lusztig kernels. In this section we show that the Borel and nilpotent part of the Frobenius-Lusztig kernels are wild except for the two known representation-finite cases $u_\zeta(\mathfrak{b}_{\mathfrak{sl}_2})$ and $u_\zeta(\mathfrak{n}_{\mathfrak{sl}_2})$. The algebra $u_\zeta(\mathfrak{n}_{\mathfrak{sl}_2})$ is just $k[X]/X^\ell$, which is a Nakayama algebra (i.e. all projective and injective indecomposables are uniserial) with $\ell$ indecomposable modules. Its smash product algebra with $u_\zeta(\mathfrak{t})\cong k\mathbb{Z}_\ell$, $u_\zeta(\mathfrak{b}_{\mathfrak{sl}_2})$, is also a Nakayama algebra with $\ell^2$ indecomposable modules (cf. \cite[2.4 (i)]{RR85})

\begin{lem}\label{cxh0+}
Let $\ell>1$ be an odd integer not divisible by $3$ if $\Phi$ is of type $\mathbb{G}_2$ and let $\zeta$ be a primitive $\ell$-th root of unity. Then $\cx_{u_\zeta(\mathfrak{b})}k\geq |\Phi^+|-|\Phi_0^+|$.
\end{lem}

\begin{proof}
As noted in \cite[Lemma 4.1]{FW09} this can be proven by following the steps in \cite{VIGRE07}. It is not necessary to assume that $k=\mathbb{C}$ as Feldvoss and Witherspoon do.
\end{proof}

\begin{thm}[{\cite[Theorem 4.3]{FW09}}, cf. {\cite[Proposition 3.3]{Cib97}} for $\mathfrak{g}$ simply-laced]\label{borelwild}
Let $\ell>1$ be an odd integer not divisible by $3$ if $\Phi$ is of type $\mathbb{G}_2$. Then $\cx_{u_\zeta(\mathfrak{b})} k\geq 3$ for $\mathfrak{g}\neq \mathfrak{sl}_2$. In particular the connected algebra $u_\zeta(\mathfrak{b})$ has wild representation type except for $\mathfrak{g}=\mathfrak{sl}_2$.
\end{thm}

\begin{proof}
The proof in \cite{FW09} applies verbatim if one observes the validity of Lemma \ref{cxh0+} in this case.
\end{proof}

From this result we are able to prove a first step towards the determination of the representation type of $U_\zeta(B_r)$ for higher $r$:

\begin{cor}\label{reptypeBr1}
Let $\ell>1$ be an odd integer not divisible by $3$ if $\Phi$ is of type $\mathbb{G}_2$. If $U_\zeta(B_r)$ is of tame or finite representation type then $\mathfrak{g}=\mathfrak{sl}_2$ and $r\leq 1$.
\end{cor}

\begin{proof}
The algebra $U_\zeta(B_r)$ contains $u_\zeta(\mathfrak{b})$ as a Hopf subalgebra. Therefore $\cx_{u_\zeta(\mathfrak{b})} k\leq \cx_{U_\zeta(B_r)} k$ for all $r\geq 0$ (cf. \cite[Proposition 2.1 (2)]{FW09}). The foregoing theorem shows that the former is greater than two. Thus the algebras are wild by \cite[Theorem 3.1]{FW09} since $U_\zeta(B_r)$ satisfies (fg) by \cite[Theorem 6.2.6]{Dru11}. For $\mathfrak{g}=\mathfrak{sl}_2$ note that $U_\zeta(B_r)//u_\zeta(\mathfrak{b})\cong \Dist(B_r)$. The latter is known to be wild for $r\geq 2$: By \cite[Corollary 2.6]{FV03}, we have that it is not tame. But the $r$-th Frobenius kernel of the additive group $G_{a(r)}$ is contained in $B_r$ and $\Dist(G_{a(r)})\cong k[X_1,\dots,X_r]/(X_1^p,\dots,X_r^p)$. Therefore $\cx_{\Dist(B_r)} k\geq \cx_{\Dist(G_{a(r)})}k\geq 2$. Thus $\Dist(B_r)$ is not representation-finite.
\end{proof}

For the remaining case the complexity does not provide any information. We therefore use a different tool: Let $Q$ be a finite quiver (an oriented graph). Then the path algebra $kQ$ is defined as the algebra with basis the set of paths in $Q$ and multiplication defined as the concatenation if defined, and zero in the other cases (see e.g. \cite[p. 50]{ARS95} for a more precise definition). Then a well-known theorem by Gabriel states that every algebra is Morita equivalent to the path algebra of a quiver modulo some ideal (whose generators are called relations). Because of its combinatorial nature, the representation theory of path algebras with relations is often quite well understood. We will use this in the following for the quantized analogue of \cite[Proposition 2.2]{FV03}:

\begin{prop}
Let $\ell>1$ be an odd integer. For $\mathfrak{g}=\mathfrak{sl}_2$, the algebra $U_\zeta(B_1)$ has the following quiver and relations: It has vertices labelled by $(i,j)$, where $i\in \mathbb{Z}_\ell$ and $j\in \mathbb{Z}_p$, arrows $(i,j)\stackrel{a_{ij}}{\to} (i+1,j+\xi)$ and $(i,j)\stackrel{b_{ij}}{\to} (i,j+1)$ and relations $b_{i+1,j+\xi}a_{ij}-a_{i,j+1}b_{ij}$ and $a_{i+\ell-1,j+(\ell-1)\xi}\cdots a_{ij}$ and $b_{i+(p-1),j+p-1}\cdots b_{ij}$.
\end{prop}

\begin{proof}
We regard $U_\zeta(N_1)$ as a module for the adjoint action of $U_\zeta(T_1)$. We have the following equalities and use the relations given in \cite[Theorem 9.3.4]{CP94} (the numbers in brackets refer to the number of the equation in this Theorem) and \cite[Proof of Proposition 11.2.4]{CP94}:
$\Ad(K)(E)\stackrel{(8)}{=}\zeta^2E$, $\Ad(K)(E^{(\ell)})\stackrel{(8)}{=}E^{(\ell)}$ and
\begin{align*}
\Ad(\begin{bmatrix}K\\\ell\end{bmatrix})(E^{(\ell)})&=\sum_{s=0}^{\ell}\begin{bmatrix}K\\\ell-s\end{bmatrix}K^{-s}E^{(\ell)}S(K^{\ell-s}\begin{bmatrix}K\\s\end{bmatrix})\stackrel{(8+9)}{=}E^{(\ell)}\sum_{s=0}^{\ell}\begin{bmatrix}K;2\ell\\\ell-s\end{bmatrix}K^{-s}S(K^{\ell-s}\begin{bmatrix}K\\s\end{bmatrix})\\
&\stackrel{(5)}=E^{(\ell)}\sum_{s=0}^{\ell}\sum_{t=0}^{\ell-s}\zeta^{2\ell(\ell-s-t)}\begin{bmatrix}2\ell\\t\end{bmatrix}K^{-t}\begin{bmatrix}K\\\ell-s-t\end{bmatrix}K^{-s}S(K^{\ell-s}\begin{bmatrix}K\\s\end{bmatrix})\\
&=E^{(\ell)}\sum_{t=0}^{\ell}\sum_{s=0}^{\ell-t}\begin{bmatrix}2\ell\\t\end{bmatrix}K^{-t}\begin{bmatrix}K\\\ell-s-t\end{bmatrix}K^{-s}S(K^{\ell-s}\begin{bmatrix}K\\s\end{bmatrix})\\
&=E^{(\ell)}\sum_{t=0}^{\ell}\begin{bmatrix}2\ell\\ t\end{bmatrix}K^{-t}\varepsilon(\begin{bmatrix}K\\\ell-t\end{bmatrix})=\begin{bmatrix}2\ell\\\ell\end{bmatrix}E^{(\ell)}
\end{align*}
The last two equations use the fact that $\varepsilon(a)=\sum_{(a)}a_{(1)}S(a_{(2)})$ and that $\begin{bmatrix}2\ell\\t\end{bmatrix}=0$ for $t<\ell$.\\
For the action of $\begin{bmatrix}K\\\ell\end{bmatrix}$ on $E$ observe that since $u_\zeta(\mathfrak{b})$ is a normal Hopf subalgebra of $U_\zeta(B)$, we have that $\Ad(\begin{bmatrix}K\\\ell\end{bmatrix})(E)\in u_\zeta(\mathfrak{b})$. Furthermore since $U_\zeta(N_1)$ is a $U_\zeta(T_1)$-module via the adjoint action, we have that $\Ad(\begin{bmatrix}K\\\ell\end{bmatrix})(E)\in U_\zeta(N_1)$. But $U_\zeta(N_1)\cap u_\zeta(\mathfrak{b})=u_\zeta(\mathfrak{n})$. Now the $\mathbb{Z}$-grading of $U_\zeta(B)$ induced by $E\mapsto 1$ and $K\mapsto 0$ yields that $\Ad(\begin{bmatrix}K\\\ell\end{bmatrix})(E)=\xi' E$ for some $\xi'\in k\setminus \{0\}$.  \\
Therefore $\langle E\rangle_k$ and $\langle E^{(\ell)}\rangle_k$ are simple modules with respect to the action of $U_\zeta(T_1)$. Since the character group is $X_{p\ell}\cong \mathbb{Z}_\ell\times \mathbb{Z}_p$, we identify the set of characters with pairs $(i,j)$, where $i\in \mathbb{Z}_\ell$ and $j\in \mathbb{Z}_p$, such that addition with the character corresponding to $E^{(\ell)}$ amounts to addition of $1$ in the second component (resp. addition with the character corresponding to $E$  amounts to addition of $1$ in the first component and with $\xi$ in the second component, where $\zeta^{\xi}=\xi'$). Now decompose $U_\zeta(T_1)=\bigoplus ku_{\lambda}$ with primitive orthogonal idempotents $u_{\lambda}$. Accordingly we have $au_{\lambda}=\lambda(a)u_{\lambda}$ for all $a\in U_\zeta(T_1)$. Hence $\lambda(u_{\lambda'})=\delta_{\lambda,\lambda'}$ for $\lambda, \lambda'\in X_{p\ell}$.\\
Since $U_\zeta(T_1)$ is semisimple and commutative, any finite dimensional module is the direct sum of $1$-dimensional simple modules, determined by characters. For a finite dimensional module $M$, let $M_\alpha$ be the sum of all submodules isomorphic to the $1$-dimensional module corresponding to the character $\alpha$. Now let $x\in U_\zeta(N_1)_\alpha$. Then we have:
\[u_\mu xu_\lambda=\sum_{(u_\mu)}\Ad(u_{\mu,(1)})(x)u_{\mu,(2)}u_\lambda=\sum_{(u_\mu)}\alpha(u_{\mu,(1)})x\lambda(u_{\mu,(2)})u_\lambda=(\lambda+\alpha)(u_\mu)xu_\lambda\\
=\delta_{\lambda+\alpha,\mu}xu_\lambda.\]
Let $Q$ be the quiver described in the statement of the proposition. Since the simple $U_\zeta(T_1)$-modules are of form $k_\lambda$ with $\lambda\in X_{p\ell}$, it follows from the foregoing equation, that there is a unique homomorphism $\Gamma: kQ\to U_\zeta(B_1)$ satisfying $\Gamma(e_\lambda)=u_\lambda$, $\Gamma(a_{ij})=Eu_{ij}$ and $\Gamma(b_{ij})=E^{(\ell)}u_{ij}$. Since $Eu_\lambda, E^{(\ell)}u_\lambda\in \image \Gamma$ and $E=\sum_{\lambda}Eu_\lambda$ and $E^{(\ell)}=\sum_{\lambda}E^{(\ell)}u_\lambda$, we have that $\Gamma$ is surjective. Furthermore we have $\Gamma(b_{i+1,j+\xi}a_{ij}-a_{i,j+1}b_{ij})=E^{(\ell)}u_{i+1,j}Eu_{ij}-Eu_{i,j+1}E^{(\ell)}u_{i,j}=0$. Since $E$ has order $\ell$ (resp. $E^{(\ell)}$ has order $p$) we as well have that multiplications of $Eu_\lambda$ of length $\ell$ (resp. $E^{(\ell)}u_\lambda$ of length $p$) are zero. Therefore if we let $I$ be the ideal generated by the elements $b_{i+1,j+\xi}a_{ij}-a_{i,j+1}b_{ij}$ and $a_{i+\ell-1,j+(\ell-1)\xi}\cdots a_{ij}$ and $b_{i,j+p-1}\cdots b_{ij}$, then $\Gamma$ factors through a surjective morphism $\hat{\Gamma}:kQ/I\to U_\zeta(B_1)$.\\
Now let $\pi: kQ\to kQ/I$ be the canonical projection. We have $[\sum_{i,j} a_{ij},\sum_{i,j}b_{ij}]\in I$ and $(\sum a_{ij})^\ell\in I$ and $(\sum b_{ij})^p\in I$. Therefore there exists a unique algebra homomorphism $\Gamma': U_\zeta(N_1)\to kQ/I$, such that $\Gamma'(E)=\pi(\sum a_{ij})$ and $\Gamma'(E^{(\ell)})=\pi(\sum b_{ij})$. We now extend $\Gamma'$ to a linear map $\hat{\Gamma}'$ on $U_\zeta(B_1)=U_\zeta(N_1)\# U_\zeta(T_1)$ via $\hat{\Gamma}'(a\# u_\lambda):=\Gamma'(a)\pi(e_\lambda)$ for all $a\in U_\zeta(N_1)$ and all $\lambda$. Direct computation shows that this is a homomorphism of $k$-algebras.\\
Let $J$ be the Jacobson radical of $kQ/I$. Then $a_{ij}=e_{i+1,j}\sum_{j'} a_{ij'}$ and $b_{ij}=e_{i,j+1}\sum_{j'} b_{ij'}$ yields $kQ/I=\image \hat{\Gamma}'+J^2$. Now apply \cite[Proposition 1.2.8]{B95} to see that $\hat{\Gamma}'$ is surjective. As a result the two algebras have the same dimension. Therefore the two given maps are isomorphisms.
\end{proof}

\begin{thm}
Let $\ell>1$ be an odd integer not divisible by $3$ for $\mathbb{G}_2$. The connected algebra $U_\zeta(B_r)$ is wild except for $r=0$ and $\mathfrak{g}=\mathfrak{sl}_2$.
\end{thm}

\begin{proof}
If $r=1$ and $\mathfrak{g}=\mathfrak{sl}_2$, the algebra contains the following subquiver with relations, indicated by the dotted lines, which is wild (see \cite{Ung90}):
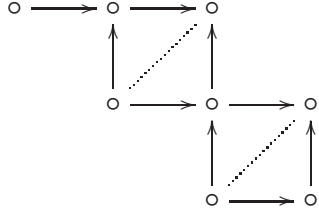
\[\begin{xy}
\xymatrix{
\circ\ar[r]&\circ\ar[r]&\circ\\
&\circ\ar[u]\ar[r]\ar@{.}[ru]&\circ\ar[u]\ar[r]&\circ\\
&&\circ\ar[u]\ar[r]\ar@{.}[ru]&\circ\ar[u]
}
\end{xy}\]
\end{proof}

The following corollary was obtained independently by Feldvoss and Witherspoon (\cite[Theorem 5.4]{FW11}) for $r=0$ by using a variant of Farnsteiner's Theorem for Hochschild cohomology (\cite[Theorem 2.1]{FW11}).

\begin{cor}
Let $\ell>1$ be an odd integer not divisible by $3$ for $\mathbb{G}_2$. The local algebra $U_\zeta(N_r)$ is wild except for $r=0$ and $\mathfrak{g}=\mathfrak{sl}_2$.
\end{cor}

\begin{proof}
We have that $U_\zeta(N_r)$ is a finite dimensional algebra, which is a $U_\zeta(T_r)$-module algebra, and $U_\zeta(B_r)=U_\zeta(N_r)\# U_\zeta(T_r)$. Therefore the wildness of $U_\zeta(B_r)$ implies the wildness of $U_\zeta(N_r)$ by \cite[Lemma 5.2.1]{Fa06} since $U_\zeta(T_r)$ is semisimple.
\end{proof}

\section{Representation type of blocks of quantum groups}

This section provides a proof of the main result of the paper, the determination of the representation type of the blocks of Frobenius-Lusztig kernels.

\begin{prop}[{\cite[Proposition 8.2.1]{BNPP11}}, {\cite[Theorem 6.1]{Dru10}}]\label{cxh0}
Let $\mathfrak{g}$ be a finite dimensional complex simple Lie algebra. Let $\ell>1$ be an odd integer not divisible by $3$ if $\Phi$ is of type $\mathbb{G}_2$. Furthermore for $\charac k=0$ assume that $\ell$ is good for $\Phi$ (i.e. $\ell\geq 3$ for type $\mathbb{B}_n$, $\mathbb{C}_n$ and $\mathbb{D}_n$, $\ell\geq 5$ for type $\mathbb{E}_6$, $\mathbb{E}_7$ and $\mathbb{G}_2$ and $\ell\geq 7$ for $\mathbb{E}_8$) and that $\ell>3$ if $\Phi$ is of type $\mathbb{B}_n$ or $\mathbb{C}_n$. For $\charac k=p>0$ assume that $p$ is good for $\Phi$ and that $\ell>h$. Let $\zeta$ be a primitive $\ell$-th root of unity and let $\lambda\in X_\ell^+$, then $\cx_{u_\zeta(\mathfrak{g})} H^0(\lambda)\geq 2\cx_{u_\zeta(\mathfrak{b})} H^0(\lambda)\geq |\Phi|-|\Phi_\lambda|$.
\end{prop}

\begin{thm}\label{quantumwild}
Let $\mathfrak{g}$ be a finite dimensional complex simple Lie algebra. Let $\ell>1$ be an odd integer not divisible by $3$ if $\Phi$ is of type $\mathbb{G}_2$. Furthermore for $\charac k=0$ assume that $\ell$ is good for $\Phi$ and that $\ell>3$ if $\Phi$ is of type $\mathbb{B}_n$ or $\mathbb{C}_n$. For $\charac k=p>0$ assume that $p$ is good for $\Phi$ and that $\ell>h$. Then any block of $u_\zeta(\mathfrak{g})$ different from the Steinberg block is of wild representation type unless $\mathfrak{g}=\mathfrak{sl}_2$.
\end{thm}

\begin{proof}
Assume to the contrary that a block of $u_\zeta(\mathfrak{g})$ is tame or representation-finite. This in particular implies that $\cx_{u_\zeta(\mathfrak{g})} H^0(\lambda)\leq 2$ for all $\lambda\in X_\ell^+$ corresponding to this block. This implies that $\cx_{u_\zeta(\mathfrak{b})}H^0(\lambda)\leq\frac{1}{2}\cx_{u_\zeta(\mathfrak{g})}H^0(\lambda)\leq 1$ by Proposition \ref{cxh0}. By Lemma \ref{cxh0+} this leads to $|\Phi^+\setminus \Phi_\lambda^+|\leq 1$. Since $\mathfrak{g}$ is simple, $\Phi_\lambda^+$ then spans $\Phi^+$ unless $\mathfrak{g}=\mathfrak{sl}_2$ (cf. \cite[Corollaire après VI.1.3 Théorème 1 (ii)]{B02}). The following easy calculation shows that with $\alpha_1, \alpha_2\in \Phi^+_\lambda$ and $a_1,a_2\in \mathbb{Z}$ such that $a_1\alpha_1+a_2\alpha_2\in \Phi^+$, also $a_1\alpha_1+a_2\alpha_2\in \Phi^+_\lambda$: Let $(\lambda+\rho, \alpha_1^\vee)=r\ell$ and $(\lambda+\rho,\alpha_2^\vee)=s\ell$, then
\begin{align*}
(\lambda+\rho, (a_1\alpha_1+a_2\alpha_2)^\vee)&=(\lambda+\rho, \frac{2(a_1\alpha_1+a_2\alpha_2)}{(a_1\alpha_1+a_2\alpha_2,a_1\alpha_1+a_2\alpha_2)})\\
&=\frac{a_1(\alpha_1,\alpha_1)}{(a_1\alpha_1+a_2\alpha_2,a_1\alpha_1+a_2\alpha_2)}r\ell+\frac{a_2(\alpha_2,\alpha_2)}{(a_1\alpha_1+a_2\alpha_2,a_1\alpha_1+a_2\alpha_2)}s\ell\\
&=\frac{ra_1\cdot (\alpha_1,\alpha_1)+sa_2\cdot (\alpha_2,\alpha_2)}{(a_1\alpha_1+a_2\alpha_2,a_1\alpha_1+a_2\alpha_2)}\cdot \ell\equiv 0\mod \ell,
\end{align*}
The last congruence follows from the fact that the whole expression is an integer and $\ell$ is odd and $3\nmid \ell$ for $\mathbb{G}_2$ and there are at most two root length differing by a factor of $2$ or $3$ for $\mathbb{G}_2$ (cf. \cite[Proposition 12 (i)]{B02}). Therefore the first factor already has to be an integer. Thus we have $\Phi_\lambda^+=\Phi^+$. This in particular implies that the simple roots $\alpha_1,\dots,\alpha_n$ are contained in $\Phi_\lambda^+$. But $( \rho, \alpha_i^\vee)=1$ for any simple root $\alpha_i$. Let $\lambda_1,\dots,\lambda_n$ be the fundamental dominant weights corresponding to $\alpha_1,\dots,\alpha_n$. We have $(\lambda_i,\alpha_j^\vee)=\delta_{ij}$. Therefore $\lambda=(\ell -1)\lambda_i\mod \ell \lambda_i, \lambda_j (j\neq i)$ for any $1\leq i\leq n$. Thus we have $\lambda=(\ell-1)\rho$. This means that the block is the Steinberg module.
\end{proof}

The representation theory of blocks of $u_\zeta(\mathfrak{sl}_2)$ is known. All blocks except the Steinberg are Morita equivalent and the quiver and relations are known (cf. \cite{S94}, \cite{Xia97}, \cite{CPm94}). These blocks satisfy very restrictive properties. To list them, let us recall some more definitions:\\

A representation-infinite algebra was said to be tame if for each dimension there is a finite number of one-parameter families classifying all modules of this dimension. A tame algebra is called domestic if there is a common bound for this finite number over all dimensions.\\

A particular class of tame (not necessarily domestic) algebras is provided by the following: An algebra $A$ is called special biserial, if $A$ is Morita equivalent to a bound quiver algebra $kQ/I$, where $Q$ and $I$ satisfy the following conditions:
\begin{enumerate}
\item[(SB1)] In each point of the quiver there start at most two arrows.
\item[(SB1')] In each point of the quiver there end at most two arrows.
\item[(SB2)] Whenever there are arrows $\alpha,\beta,\gamma$ with $s(\alpha)=s(\beta)=t(\gamma)$, we have $\alpha\gamma\in I$ or $\beta\gamma\in I$.
\item[(SB2')] Whenever there are arrows $\alpha,\beta,\gamma$ with $t(\alpha)=t(\beta)=s(\gamma)$, we have $\gamma\alpha\in I$ or $\gamma\beta \in I$.
\end{enumerate}

This class of algebras is well-studied, e.g. a complete classification of the indecomposable modules and the Auslander-Reiten quiver is known (see \cite{WW85}).\\

In our case we have the following properties:

\begin{cor}
Impose the same restriction on the parameters as in the foregoing theorem. Then the following are equivalent for a block $B$ of a small quantum group $u_\zeta(\mathfrak{g})$:
\begin{enumerate}[(1)]
\item $B$ is tame.
\item $B$ is domestic.
\item $B$ is special biserial.
\item $B$ is Morita equivalent to the trivial extension of the path algebra of the Kronecker quiver by the dual bimodule.
\item There exists an induced module $H^0(\lambda)$ for $B$ with $\cx_B H^0(\lambda)=2$.
\item There exists a simple $B$-module $S$ such that $\cx_B S=2$
\end{enumerate}
\end{cor}

\begin{proof}
The equivalence of the first five statements follows easily from the representation theory of $u_\zeta(\mathfrak{sl}_2)$. For $\ell>h$ and $(l,r)=1$ for all bad primes $r$ of $\Phi$ the equivalence to the last statement follows from the explicit computation of the support variety of the simples in \cite[Theorem 3.3]{DNP10}. But the statement is true more generally by a quantum analogue of \cite[Proposition 5.1]{Fa00b}, that we will prove in a forthcoming paper.
\end{proof}

\begin{lem}
Let $\ell>1$ be an odd root of unity and let $\charac k\geq 3$. For $\lambda\neq \mu\in X_\ell$ let $L(\lambda), L(\mu)$ be two simple $u_\zeta(\mathfrak{sl}_2)$-modules belonging to the same block. Then $\Ext^1_{u_\zeta(\mathfrak{sl}_2)}(L(\lambda),L(\mu))\cong L(1)$ as $\Dist(\SL(2)_r)$-modules.
\end{lem}

\begin{proof}
It is known that the projective $u_\zeta(\mathfrak{sl}_2)$-module $P(\lambda)$ is a uniserial $U_\zeta(\mathfrak{sl}_2)$-module with composition factors $L(\lambda),L(2\ell-2-\lambda),L(\lambda)$ in that order: Being a tilting module $P(\lambda)$ has a filtration by Weyl modules, where the Weyl module to the highest weight $\lambda$ occurs exactly once. Furthermore $P(\lambda)|_{u_\zeta(\mathfrak{sl}_2)}$ has composition factors of weight $(\lambda, \ell-2-\lambda, \ell-2-\lambda, \lambda)$. The Weyl modules corresponding to the restricted weights are known to be simple, the others are uniserial of length $2$, see e.g. \cite[Example 11.2.7]{CP94}. But as the Weyl module for $\lambda$ occurs exactly once, the Weyl filtration has to have factors $W(2\ell-2-\lambda)$ (with simple top $L(2\ell-2-\lambda)$ ) and $W(\lambda)$ (see also \cite[p. 202-203]{CK05} for a computational argument in characteristic zero). The sequence $0\to W(2\ell-2-\lambda)\to P(\lambda)\to L(\lambda)\to 0$ induces a long exact $\Ext$-sequence, which is a sequence of $\Dist(\SL(2)_r)$-modules (the action of $\Dist(\SL(2)_r)$ on the $\Ext$-groups is defined in the proof of Lemma \ref{extsimpleextensionHopf}):
\begin{align*}
0&\to \Hom_{u_\zeta(\mathfrak{sl}_2)}(L(\lambda),L(\mu))\to \Hom_{u_\zeta(\mathfrak{sl}_2)}(P(\lambda),L(\mu))\to \Hom_{u_\zeta(\mathfrak{sl}_2)}(W(2\ell-2-\lambda),L(\mu))\\
&\to \Ext^1_{u_\zeta(\mathfrak{sl}_2)}(L(\lambda),L(\mu))\to \Ext^1_{u_\zeta(\mathfrak{sl}_2)}(P(\lambda),L(\mu)).
\end{align*}
Since $P(\lambda)$ is the $u_\zeta(\mathfrak{sl}_2)$-projective cover of $L(\lambda)$ we have that \[\Ext^1_{u_\zeta(\mathfrak{sl}_2)}(L(\lambda),L(\mu))\cong \Hom_{u_\zeta(\mathfrak{sl}_2)}(W(2\ell-2-\lambda),L(\mu)).\]
But $\Kopf_{u_\zeta(\mathfrak{sl}_2)} W(2\ell-2-\lambda)=\Kopf_{U_\zeta(\SL(2)_r)} W(2\ell-2-\lambda)=L(1)^{[1]}\otimes L(\mu)$, where $L(1)^{[1]}$ denotes the twodimensional simple $\Dist(\SL(2)_r)$-module $L(1)$ viewed as a $U_\zeta(\SL(2)_r)$-module. Thus the canonical projection $W(2\ell-2-\lambda)\to \Kopf W(2\ell-2-\lambda)$ is a $U_\zeta(\SL(2)_r)$-module homomorphism. Hence $\Hom_{u_\zeta(\mathfrak{sl}_2)}(\Kopf W(2\ell-2-\lambda),L(\mu))\to \Hom_{u_\zeta(\mathfrak{sl}_2)}(W(2\ell-2-\lambda),L(\mu))$ is a $U_\zeta(\SL(2)_r)$-module monomorphism on which $u_\zeta(\mathfrak{sl}_2)$ acts trivially, i.e. a $\Dist(\SL(2)_r)$-module isomorphism, since they have the same dimension. But $\Hom_{u_\zeta(\mathfrak{sl}_2)}(L(1)^{[1]}\otimes L(\mu),L(\mu))\cong L(1)$.
\end{proof}

The foregoing lemma can be used to compute the quiver of $U_\zeta(\SL(2)_1)$ by applying Lemma \ref{extsimpleextensionHopf}. To determine the representation type of the algebra it is only necessary to compute a small part of it. This will be done in the next statement.

\begin{prop}
Impose the same restriction on the parameters as in the foregoing theorem. The algebra $U_\zeta(G_r)$ is of wild representation type except for $r=0$ and $G=\SL(2)$.
\end{prop}

\begin{proof}
Using the fact that $U_\zeta(G_r)//u_\zeta(\mathfrak{g})\cong \Dist(G_r)$ one concludes that $G=\SL(2)$ and $r\leq 1$ since if these assumptions do not hold, $\Dist(G_r)$ is known to be wild.\\
For $r=1$ we have by Lemma \ref{extsimpleextensionHopf} that if none of the involved simples is the Steinberg module for the corresponding algebra, then for given $\mu'$ there exists $\mu$, such that we have that $\dim \Ext^1_{U_\zeta(\SL(2)_r)}(L(\mu)^{[1]}\otimes L(\lambda'), L(\mu')^{[1]}\otimes L(\lambda'))=2$ for all $\lambda$.
Furthermore we have \[\dim \Ext^1(L(\mu)^{[1]}\otimes L(\lambda),L(\mu')^{[1]}\otimes L(\lambda'))=\delta_{\lambda',\ell-2-\lambda}(\delta_{\mu',\mu-1}+\delta_{\mu',\mu+1})\] for $\mu'\neq p-1$ by \cite[Kapitel 5, Satz]{Fis82} and Lemma \ref{extsimpleextensionHopf}. Thus the quiver of $U_\zeta(\SL(2)_1)$ contains a subquiver of the form
$\begin{xy}\xymatrix{\circ\ar@<2pt>[r]\ar@<-2pt>[r]&\circ&\circ\ar[l]}\end{xy}$. Therefore the algebra is wild by \cite[I.10.8 (iii)]{Er90}.
\end{proof}

We can even show more in the case that $U_\zeta(G_r)$ satisfies (fg), which is only proven in some instances in \cite[Theorem 5.3.4]{Dru11} and conjectured to be true in much more generality (cf. \cite[Conjecture 2.18]{EO04}).

\begin{prop}
Assume that $U_\zeta(G_r)$ satisfies (fg). If there is a tame block, then $\mathfrak{g}=\mathfrak{sl}_2$ and either $r=0$ or it is induced by a block of $\Dist(\SL(2)_1)$, i.e. the corresponding simple modules have weight $(\ell-1+p\mu)$.
\end{prop}

\begin{proof}
Restricting the induced modules to $u_\zeta(\mathfrak{g})$, we see that their complexity is too big in case $\mathfrak{g}\neq \mathfrak{sl}_2$. In the case of $\mathfrak{g}=\mathfrak{sl}_2$ the wildness of the remaining blocks follows from the proof of the foregoing proposition if we observe that the argument in the foregoing proof already shows that each of the simple modules of the form $L(p-1)^{[1]}\otimes L(\lambda)$ corresponds to one of the blocks, that we have already shown to be wild for $\lambda\neq \ell-1$.
\end{proof}

\section*{Acknowledgement}
The results of this article are part of my doctoral thesis, which I am currently writing at the University of Kiel. I would like to thank my advisor Rolf Farnsteiner for his continuous support, especially for answering my questions on the classical case. I also would like to thank Chris Drupieski for answering some of my questions on Frobenius-Lusztig kernels, especially on restrictions of the parameter. Furthermore I thank the members of my working group for proof reading.

\bibliographystyle{alpha}
\bibliography{publication}

\end{document}